\UseRawInputEncoding
\documentclass[11pt, eqno]{article}
\usepackage{bbm}
\usepackage{mathrsfs}
\usepackage{amsfonts}
\usepackage{amssymb}
\usepackage{graphicx}
\usepackage[all]{xy}
\usepackage{amsthm}
\usepackage{amsmath}
\usepackage{amsmath,amssymb,latexsym,color}
\usepackage[mathscr]{eucal}
\usepackage{CJK}
\usepackage{cases}
\usepackage{graphics}

\usepackage{psfrag}
\usepackage{subfigure}
\usepackage{color}

\usepackage{amssymb,latexsym}
\usepackage{amsmath,latexsym}
\usepackage{amscd}

\newtheorem{theorem}{Theorem}[section]

\newtheorem{corollary}[theorem]{Corollary}

\newtheorem{remark}[theorem]{Remark}

\newtheorem{lemma}[theorem]{Lemma}

\newtheorem{guess}[theorem]{Conjecture}

\begin{document}

\title{Ekeland-Hofer-Zehnder capacities of lagrangian products with special forms}
\date{May, 2024}
\author{Kun Shi
\thanks{Corresponding author
\endgraf\hspace{2mm} 2020 {\it Mathematics Subject Classification.}
 53D05, 53C23 (primary), 70H05, 57R17 (secondary).}}
 \maketitle \vspace{-0.3in}



\abstract{
In this paper, we give some estimations for Ekeland-Hofer-Zehnder capacities of lagrangian products with special forms through combinatorial formulas. Based on these estimations, we give some interesting corollaries. } \vspace{-0.1in}
\medskip\vspace{12mm}

\maketitle \vspace{-0.5in}


\tableofcontents

\section{Introduction and results}
\setcounter{equation}{0}

Recently, P. Haim-Kislev \cite{PH19} introduced a combinatorial formula for the Ekeland-Hofer-Zehnder capacity of a convex polytope in $\mathbb{R}^{2n}$ and proved a certain subadditivity property of this capacity as an application of this formula. Prompted by this, we give some estimations for Ekeland-Hofer-Zehnder capacity of special convex body through combinatorial formula. We also have some result obtained by combination formula which can be found in \cite{KS24, KS241}.

Let $\omega_0$ be the standard symplectic form on $\mathbb{R}^{2n}$ and $J_{2n}$ be the standard complex structure on $\mathbb{R}^{2n}$. Our first result is as following.
\begin{theorem}\label{th:JK}
Let $K\subset(\mathbb{R}^{2n}, \omega_0)$ be a convex body, then
$$c_{\rm EHZ}(K\times_L J_{2n}K)\leqslant 2c_{\rm EHZ}(K).$$
\end{theorem}
 If $n=1$, then we can get the following corollary.
\begin{corollary}\label{cor:1}
Let $K\subset(\mathbb{R}^{2}, \omega_0)$ be a convex body, then
$$c_{\rm EHZ}(K\times_L J_2K)\leqslant 2{\rm Vol}(K), $$
where ${\rm Vol}(K)$ denotes the Euclidean volume of $K$.
\end{corollary}

About the relationship between symplectic capacities and volume of convex bodies, Viterbo\cite{Vi00} proposed an influential conjecture in symplectic geometry, which is called Viterbo¡¯s volume-capacity conjecture.

\begin{guess}[Viterbo \cite{Vi00}]\label{conj:viterbo}
{\rm On $\mathbb{R}^{2n}$, for any normalized symplectic capacity $c$ and any bounded convex domain $D$ there holds
\begin{equation}\label{e:viterbo}
\frac{c(D)}{c(B^{2n}(1))}\le \left(\frac{{\rm Vol}(D)}{{\rm Vol}(B^{2n}(1))}\right)^{1/n}
\end{equation}
(or equivalently $(c(D))^n\le n!{\rm Vol}(D)$),
with equality if and only if $D$ is symplectomorphic to the Euclidean ball. }
\end{guess}

\begin{remark}
{\rm In \cite{HO24} P. Haim-Kislev and Y. Ostrover give a counterexample to Viterbo's volume-capacity conjecture. This implies that $1$ is not the uper bound of
$$\frac{c(K)}{({n!{\rm Vol}(K)})^{1/n}}. $$
Our Corollary~\ref{cor:1} gives a upper bound $\sqrt{2}$ of $$\frac{c(K)}{\sqrt[1/n]{n!{\rm Vol}(K)}}$$ under the condition of Corollary~\ref{cor:1}.
}
\end{remark}

Our second result is about special convex polytopes with form $K\times_L K$.
\begin{theorem}\label{th:K}
If $K\subset(\mathbb{R}^{2n}, \omega_0)$ is a convex polytope with $(2n-1)$-dimensional facets $\{F_i\}_{i=1}^N$,
   $n_i$ is the unit outer normal to $F_i$, and  $h_i = h_K(n_i)$ the ``oriented height" of $F_i$ given by
  the support function of $K$, $h_K(y) := \sup_{x \in K} \langle x,y \rangle$. Then,
  $$c_{\rm EHZ}(K\times_L K)\leqslant2\sum_{i=1}^{N}
h_{i}^2.$$

\end{theorem}

In fact, the Ekeland-Hofer-Zehnder capacities are invariant under translation. Thus, we can make any interior point be the coordinate origin. In this case, $h_i$ is the distance from this interior point to $F_i$. Therefore, Theorem~\ref{th:K} can imply $c_{\rm EHZ}(K\times_L K)$ may be controlled by twice the minimum value of the sum of squared distances from any point within this polytope to each face. This result is stated as follows.

\begin{corollary}
If $K\subset(\mathbb{R}^{2n}, \omega_0)$ is a convex polytope with $(2n-1)$-dimensional facets $\{F_i\}_{i=1}^N$. Then,
\begin{equation}\label{e:Con}
c_{\rm EHZ}(K\times_L K)\leqslant2\min_{x\in K}\sum_{i=1}^{N}\left({\rm dist}(x, F_i)\right)^2, 
\end{equation}
where ${\rm dist}(x, F_i)$ is the Euclidean distance from $x$ to the hyperplane containing the face $F_i$.
\end{corollary}

The left of (\ref{e:Con}) is a question in the field of convex optimization.  

This paper is organized as follows.
In the next section we prove Theorem~\ref{th:JK}. Then we shall prove Theorem~\ref{th:K} in Section~\ref{sec:3}.

\section{Proof of Theorem~\ref{th:JK}\label{sec:2}}

Prompted by Gromov's seminal work \cite{Gr85}, Ekeland and Hofer  \cite{EH89}
defined a  \textsf{symplectic capacity} on the $2n$-dimensional Euclidean space $\mathbb{R}^{2n}$
with the standard symplectic structure $\omega_0$
to be a map $c$ which associates to each subset $U\subset \mathbb{R}^{2n}$ a number
 a number $c(U)\in[0,\infty]$  satisfying the following axioms:
\begin{description}
\item[(Monotonicity)] $c(U)\le c(V)$ for $U\subset V\subset\mathbb{R}^{2n}$;
\item[(Conformality)] $c(\psi(U))=|\alpha|c(U)$ for $\psi\in{\rm Diff}(\mathbb{R}^{2n})$
such that $\psi^\ast\omega_0=\alpha\omega_0$ with $\alpha\ne 0$;
\item[(Nontriviality)] $0<c\big(B^{2n}(1)\big)$ and $c\big(Z^{2n}(1)\big)<\infty$,
where $B^{2n}(r)=\{z\in\mathbb{R}^{2n}\,|\, |z|^2<r^2\}$ and $Z^{2n}(R)=B^2(R)\times\mathbb{R}^{2n-2}$.
\end{description}
Moreover, such a symplectic capacity is called \textsf{normalized} if it also satisfies
\begin{description}
\item[(Normalization)] $c\big(B^{2n}(1)\big)=c\big(Z^{2n}(1)\big)=\pi$.
\end{description}

Starting from the representation formula for Ekeland-Hofer-Zehnder capacities  \cite{EH89, HoZe90}, P. Haim-Kislev introduced a combinatorial formula in \cite{PH19}.

\begin{lemma}[\hbox{\cite[Throrem~1.1]{PH19}}]\label{le:EHZ}
If $K$ is a convex polytope with   $(2n-1)$-dimensional facets $\{F_i\}_{i=1}^{{\bf F}_K}$,
   $n_i$ is the unit outer normal to $F_i$, and  $h_i = h_K(n_i)$ the ``\textbf{oriented height}" of $F_i$ given by
  the support function of $K$, $h_K(y) := \sup_{x \in K} \langle x,y \rangle$,
 \begin{equation}\label{e:Comb formula}
c_{\rm EHZ}(K) = \frac{1}{2} \left[ \max_{\sigma\in S_{{\bf F}_K}, (\beta_i) \in M(K)}  \sum_{1 \leq j < i \leq {\bf F}_K}{} \beta_{\sigma(i)} \beta_{\sigma(j)} \omega_0(n_{\sigma(j)},n_{\sigma(i)}) \right]^{-1},
\end{equation}
where $S_{{\bf F}_K}$ is the symmetric group on ${\bf F}_K$ letters and
$$
M(K) = \left\{ (\beta_i)_{i=1}^{{\bf F}_K} : \beta_i \geq 0, \sum_{i=1}^{{\bf F}_K} \beta_i h_i = 1, \sum_{i=1}^{{\bf F}_K} \beta_i n_i = 0 \right\}.
$$
\end{lemma}
Let $\langle \cdot, \cdot\rangle_{\mathbb{R}^{2n}}$ be the standard Euclidean inner product in $\mathbb{R}^{2n}$. Under our convention $\langle x, y\rangle_{\mathbb{R}^{2n}}=\omega_0(x, J_{2n}y)$, $\omega_0(n_{\sigma(i)}, n_{\sigma(j)})$ in \cite[Throrem~1.1]{PH19}
should be changed into $\omega_0(n_{\sigma(j)}, n_{\sigma(i)})$.
\begin{proof}[Proof of Theorem\ref{th:JK}]
Since  Ekeland-Hofer-Zehnder symplectic capacity is continuous on the class of convex bodies with respect to the Hausdorff distance(cf. \cite{HoZe90}).
Moreover, every convex body can be approximated by convex polytopes in the Hausdorff metric by \cite[Theorem~1.8.13]{Sch93}.
\textsf{Hence from now on we may assume that $K$ is a convex polytope.}

Without loss of generality, we assume $0\in {\rm int}(K)$. Assume $K$ is a convex polytope in $(\mathbb{R}^{2n}, \omega_0)$ with $N$ faces $P_1, P_2, \cdots, P_N$.
Let $n_i$ be the unit outer normal  to $P_i$ and $h_i = h_K(n_i)$ the oriented height  of $n_i$.
The Ekeland-Hofer-Zehnder capacity of $K$ is
 $$
c_{\rm EHZ}(K)=\frac{1}{2}\min_{((\beta_i)_{i=1}^{N},\sigma)\in M(K)}\frac{1}{\sum_{1\leqslant
j<i\leqslant N}
\beta_{\sigma(i)}\beta_{\sigma(j)}\omega_0(n_{\sigma(j)},n_{\sigma(i)})},
$$
where
\begin{equation*}
 M(K)=\left\{((\beta_i)_{i=1}^{N},\sigma)\,\bigg|\,\begin{array}{ll}
&\beta_i\geqslant
0,\quad \sum_{i=1}^{N}\beta_ih_i=1,\quad \sum_{i=1}^{N}\beta_iJ_{2n} n_i=0,\\
&\sum_{1\leqslant
j<i\leqslant{N}}
\beta_{\sigma(i)}\beta_{\sigma(j)}\omega_0(n_{\sigma(j)},n_{\sigma(i)})>0,\;\sigma\in S_{N}
\end{array}
\right\}.
\end{equation*}

Next, we consider the Lagrangian product $K\times_L J_{2n}K$. $K\times_L J_{2n}K$ has $2N$ faces which are $\tilde{P}_1=P_1\times J_{2n}K, \tilde{P}_2=P_2\times J_{2n}K, \cdots, \tilde{P}_N=P_N\times J_{2n}K, \tilde{P}_{N+1}=K\times J_{2n}P_1, \tilde{P}_{N+2}=K\times J_{2n}P_2, \cdots, \tilde{P}_{2N}=K\times J_{2n}P_N$. The the unit outer normals to those faces are $\tilde{n}_1=(n_1,0), \tilde{n}_2=(n_2, 0), \cdots, \tilde{n}_N=(n_N,0), \tilde{n}_{n+1}=(0, J_{2n}n_1), \tilde{n}_{n+2}=(0, J_{2n}n_2), \cdots, \tilde{n}_{2N}=(0, J_{2n}n_N)$. The oriented height $\tilde{h}_i$ and $\tilde{h}_{N+i}$ of $\tilde{n}_i$ and $\tilde{n}_{N+i}$ both are $h_i$ for $i=1, 2, \cdots, N$.

For any $((\beta_i)_{i=1}^{N},\sigma)\in M(K)$, we define
$$\tilde{\beta}(i)=\left\{
           \begin{array}{ll}
             \frac{1}{2}\beta_i, &  \text{ if } i=1,\cdots, N;  \\
             \frac{1}{2}\beta_{i-N}, &  \text{ if } i=N+1,\cdots, 2N. \\
           \end{array}
         \right.$$
and
$$\tilde{\sigma}(i)=\left\{
           \begin{array}{ll}
             \sigma(\frac{i+1}{2})+N, &  \text{ if } i \text{ is odd};  \\
             \sigma(\frac{i}{2}), &  \text{ if } i\text{ is even}. \\
           \end{array}
         \right.$$
Thus, $((\tilde{\beta}_i)_{i=1}^{2N},\tilde{\sigma})\in M(K\times_L J_{2n}K)$. In fact, $\sum_{i=1}^{2n}\tilde{\beta}_i\tilde{h}_i=\sum_{i=1}^{N}\frac{1}{2}\beta_ih_i+\sum_{i=1}^{N}\frac{1}{2}\beta_ih_i=\sum_{i=1}^{N}\beta_ih_i=1$ and $\sum_{i=1}^{2N}\tilde{\beta}_i\tilde{n}_i=\sum_{i=1}^{N}\frac{1}{2}\beta_i (n_i, 0)+\sum_{i=1}^{N}\frac{1}{2}\beta_i (0, J_{2n}n_i)=0$.

Let $\tilde{\omega}_{0}$ be the symplectic form on $K\times_L J_{2n}K$.
Notice that
$$\tilde{\omega}_{0}((n_i,0), (0, J_{2n}n_j))=\langle n_i, J_{2n}n_j\rangle_{\mathbb{R}^{2n}}=-\omega_0(n_i, n_j) $$
and
$$\tilde{\omega}_0((n_i,0), (n_j, 0))=\tilde{\omega}_0((0, J_{2n}n_i), (0, J_{2n}n_j))=0,$$
we can deduce
\begin{eqnarray*}
\sum_{1\leqslant
j<i\leqslant{2N}}
\tilde{\beta}_{\tilde{\sigma}(i)}\tilde{\beta}_{\tilde{\sigma}}(j)\tilde{\omega}_0(n_{\tilde{\sigma}(j)}, n_{\tilde{\sigma}(i)})&=&
\sum_{1\leqslant
j\leqslant i\leqslant{N}}\frac{1}{4}\beta_{\sigma(i)}\beta_{\sigma(j)}\tilde{\omega}_{0}((0, J_{2n}n_{\sigma(j)}), (n_{\sigma(i)},0))\\
&&+\sum_{1\leqslant
j< i\leqslant{N}}
\frac{1}{4}\beta_{\sigma(i)}\beta_{\sigma(j)}\tilde{\omega}_{0}((0, J_{2n}n_{\sigma(j)}), (0, J_{2n}n_{\sigma(i)}))\\
&&+\sum_{1\leqslant
j< i\leqslant N}
\frac{1}{4}\beta_{\sigma(i)}\beta_{\sigma(j)}\tilde{\omega}_{0}((n_{\sigma(j)},0), (0, J_{2n}n_{\sigma(i)}))\\
&&+\sum_{1\leqslant
j<i\leqslant{N}}\frac{1}{4}\beta_{\sigma(i)}\beta_{\sigma(j)}\tilde{\omega}_{0}((0, J_{2n}n_{\sigma(j)}), (n_{\sigma(i)},0))\\
&=&\sum_{1\leqslant
j\leqslant i\leqslant N}
\frac{1}{4}\beta_{\sigma(i)}\beta_{\sigma(j)}\omega_0 (n_{\sigma(j)}, n_{\sigma(i)})\\
&&+\sum_{1\leqslant
j<i\leqslant{N}}\frac{1}{4}\beta_{\sigma(i)}\beta_{\sigma(j)}\omega_0 (n_{\sigma(j)}, n_{\sigma(i)})\\
&=&\sum_{1\leqslant
j\leqslant i\leqslant N}
\frac{1}{2}\beta_{\sigma(i)}\beta_{\sigma(j)}\omega_0 (n_{\sigma(j)}, n_{\sigma(i)})\\
\end{eqnarray*}
Therefore, for any  $((\beta_i)_{i=1}^{N},\sigma)\in M(K)$,
\begin{eqnarray*}
c_{\rm EHZ}(K\times_L JK)\leqslant \frac{1}{2}\frac{1}{\sum_{1\leqslant
j<i\leqslant{2N}}
\tilde{\beta}_{\tilde{\sigma}(i)}\tilde{\beta}_{\tilde{\sigma}}(j)\tilde{\omega}_0(n_{\tilde{\sigma}(j)}, n_{\tilde{\sigma}(i)})}=\frac{1}{\sum_{1\leqslant
j\leqslant i\leqslant N}
\beta_{\sigma(i)}\beta_{\sigma(j)}\omega_0 (n_{\sigma(j)}, n_{\sigma(i)})}.
\end{eqnarray*}
By Lemma~\ref{le:EHZ} and arbitrariness of $((\beta_i)_{i=1}^{N},\sigma)\in M(K)$,
$$c_{\rm EHZ}(K\times_L J_{2n}K)\leqslant 2c_{\rm EHZ}(K).$$

\end{proof}

\section{Proof of Theorem~\ref{th:K}\label{sec:3}}
\begin{proof}[Proof of Theorem\ref{th:K}]

The Ekeland-Hofer-Zehnder capacity of $K$ is
 $$
c_{\rm EHZ}(K)=\frac{1}{2}\min_{((\beta_i)_{i=1}^{N},\sigma)\in M(K)}\frac{1}{\sum_{1\leqslant
j<i\leqslant N}
\beta_{\sigma(i)}\beta_{\sigma(j)}\omega_0(n_{\sigma(j)},n_{\sigma(i)})},
$$
where
\begin{equation*}
 M(K)=\left\{((\beta_i)_{i=1}^{N},\sigma)\,\bigg|\,\begin{array}{ll}
&\beta_i\geqslant
0,\quad \sum_{i=1}^{N}\beta_ih_i=1,\quad \sum_{i=1}^{N}\beta_i J_{2n} n_i=0,\\
&\sum_{1\leqslant
j<i\leqslant{N}}
\beta_{\sigma(i)}\beta_{\sigma(j)}\omega_0(n_{\sigma(j)},n_{\sigma(i)})>0,\;\sigma\in S_{N}
\end{array}
\right\}.
\end{equation*}

Next, we consider the Lagrangian product $K\times_L K$. $K\times_L K$ has $2N$ faces which are $\tilde{F}_1=F_1\times K, \tilde{F}_2=F_2\times K, \cdots, \tilde{F}_N=F_N\times K, \tilde{F}_{N+1}=K\times F_1, \tilde{F}_{N+2}=K\times F_2, \cdots, \tilde{F}_{2N}=K\times F_N$. The the unit outer normals to those faces are $\tilde{n}_1=(n_1,0), \tilde{n}_2=(n_2, 0), \cdots, \tilde{n}_N=(n_N,0), \tilde{n}_{n+1}=(0, n_1), \tilde{n}_{n+2}=(0, n_2), \cdots, \tilde{n}_{2N}=(0, n_N)$. The oriented height $\tilde{h}_i$ and $\tilde{h}_{N+i}$ of $\tilde{n}_i$ and $\tilde{n}_{N+i}$ both are $h_i$ for $i=1, 2, \cdots, N$.

For any $((\beta_i)_{i=1}^{N},\sigma)\in M(K)$, we define
$$\tilde{\beta}(i)=\left\{
           \begin{array}{ll}
             \frac{1}{2}\beta_i, &  \text{ if } i=1,\cdots, N;  \\
             \frac{1}{2}\beta_{i-N}, &  \text{ if } i=N+1,\cdots, 2N. \\
           \end{array}
         \right.$$
and
$$\tilde{\sigma}(i)=\left\{
           \begin{array}{ll}
             \sigma(\frac{i+1}{2}), &  \text{ if } i \text{ is odd};  \\
             \sigma(\frac{i}{2})+N, &  \text{ if } i\text{ is even}. \\
           \end{array}
         \right.$$
Thus, $((\tilde{\beta}_i)_{i=1}^{2N},\tilde{\sigma})\in M(K\times_L K)$. In fact, 
$$\sum_{i=1}^{2n}\tilde{\beta}_i\tilde{h}_i=\sum_{i=1}^{N}\frac{1}{2}\beta_ih_i+\sum_{i=1}^{N}\frac{1}{2}\beta_ih_i=\sum_{i=1}^{N}\beta_ih_i=1$$ and 
$$\sum_{i=1}^{2N}\tilde{\beta}_i\tilde{n}_i=\sum_{i=1}^{N}\frac{1}{2}\beta_i (n_i, 0)+\sum_{i=1}^{N}\frac{1}{2}\beta_i (0, n_i)=0.$$

Let $\tilde{\omega}_{0}$ be the symplectic form on $K\times_L K$.
Notice that
$$\tilde{\omega}_{0}((n_i,0), (0, n_j))=\langle n_i, n_j\rangle_{\mathbb{R}^{2n}} $$
and
$$\tilde{\omega}_0((n_i,0), (n_j, 0))=\tilde{\omega}_0((0, n_i), (0, n_j))=0,$$
we can deduce
\begin{eqnarray*}
\sum_{1\leqslant
j<i\leqslant{2N}}
\tilde{\beta}_{\tilde{\sigma}(i)}\tilde{\beta}_{\tilde{\sigma}}(j)\tilde{\omega}_0(n_{\tilde{\sigma}(j)},n_{\tilde{\sigma}(i)})&=&\sum_{1\leqslant
j\leqslant i\leqslant N}
\frac{1}{4}\beta_{\sigma(i)}\beta_{\sigma(j)}\tilde{\omega}_{0}((n_{\sigma(j)},0), (0, n_{\sigma(i)}))\\
&&+\sum_{1\leqslant
j< i\leqslant{N}}
\frac{1}{4}\beta_{\sigma(i)}\beta_{\sigma(j)}\tilde{\omega}_{0}((n_{\sigma(j)},0), (n_{\sigma(i)},0))\\
&&+\sum_{1\leqslant
j<i\leqslant{N}}\frac{1}{4}\beta_{\sigma(i)}\beta_{\sigma(j)}\tilde{\omega}_{0}((0,n_{\sigma(j)}), (n_{\sigma(i)},0))\\
&&+\sum_{1\leqslant
j< i\leqslant{N}}
\frac{1}{4}\beta_{\sigma(i)}\beta_{\sigma(j)}\tilde{\omega}_{0}((0,n_{\sigma(j)}), (0, n_{\sigma(i)}))\\
&=&\sum_{1\leqslant
j\leqslant i\leqslant N}
\frac{1}{4}\beta_{\sigma(i)}\beta_{\sigma(j)}\tilde{\omega}_{0}((n_{\sigma(j)},0), (0, n_{\sigma(i)}))\\
&&+\sum_{1\leqslant
j<i\leqslant{N}}\frac{1}{4}\beta_{\sigma(i)}\beta_{\sigma(j)}\tilde{\omega}_{0}((0,n_{\sigma(j)}), (n_{\sigma(i)},0))\\
&=&\sum_{1\leqslant
j\leqslant i\leqslant N}
\frac{1}{4}\beta_{\sigma(i)}\beta_{\sigma(j)}\langle n_{\sigma(j)}, n_{\sigma(i)}\rangle_{\mathbb{R}^{2n}}\\
&&-\sum_{1\leqslant
j<i\leqslant{N}}\frac{1}{4}\beta_{\sigma(i)}\beta_{\sigma(j)}\langle n_{\sigma(j)}, n_{\sigma(i)}\rangle_{\mathbb{R}^{2n}}\\
&=&\frac{1}{4}\sum_{i=1}^{N}
\beta_{i}\beta_{i}\omega_0(n_i,  n_i)=\frac{1}{4}\sum_{i=1}^{N}
\beta_{i}^2>0.\\
\end{eqnarray*}
Since $\sum_{i=1}^{N}\beta_i h_i=1$,
\begin{eqnarray*}
\sum_{1\leqslant
j<i\leqslant{2N}}
\tilde{\beta}_{\tilde{\sigma}(i)}\tilde{\beta}_{\tilde{\sigma(j)}}\tilde{\omega}_0(n_{\tilde{\sigma}(j)},n_{\tilde{\sigma}(i)})=\frac{1}{4}\sum_{i=1}^{N}
\beta_{i}^2\geq \frac{1}{4}\frac{1}{\sum_{i=1}^{N}
h_{i}^2}
\end{eqnarray*}
by using Cauchy-Schwarz inequality.

Moreover, in term of Lemma~\ref{le:EHZ}, we can deduce
\begin{eqnarray*}
c_{\rm EHZ}(K\times_L K)&\leqslant& \frac{1}{2}\frac{1}{\sum_{1\leqslant
j<i\leqslant{2N}}
\tilde{\beta}_{\tilde{\sigma}(i)}\tilde{\beta}_{\tilde{\sigma}(j)}\tilde{\omega}_0(n_{\tilde{\sigma}(j)},n_{\tilde{\sigma}(i)})}\\
&=&\frac{1}{2}\frac{1}{\frac{1}{4}\frac{1}{\sum_{i=1}^{N}
h_{i}^2}}\\
&=&2\sum_{i=1}^{N}
h_{i}^2.\\
\end{eqnarray*}

\end{proof}

\medskip
\begin{tabular}{l}
Kun Shi \\
School of Mathematical Sciences, Tiangong University\\
Tianjin 300387, The People's Republic of China\\
E-mail address: kunshi@tiangong.edu.cn\\
\end{tabular}

\end{document}